\documentclass[11pt]{amsart}

\usepackage{amsmath}
\usepackage{amssymb}
\usepackage{latexsym}
\usepackage{amscd}
\usepackage{citesort}
\usepackage{mathrsfs}
\usepackage[all]{xy}

\newdimen\AAdi%
\newbox\AAbo%
\def\AAk#1#2{\s_etbox\AAbo=\hbox{#2}\AAdi=\wd\AAbo\kern#1\AAdi{}}%
\def\AAr#1#2#3{\s_etbox\AAbo=\hbox{#2}\AAdi=\ht\AAbo\raise#1\AAdi\hbox{#3}}%
\font\tenmsb=msbm10 at 12pt \font\sevenmsb=msbm7 at 8pt
\font\fivemsb=msbm5 at 6pt
\newfam\msbfam
\textfont\msbfam=\tenmsb \scriptfont\msbfam=\sevenmsb
\scriptscriptfont\msbfam=\fivemsb

\newtheorem{theorem}{Theorem}

\newtheorem{remark}[theorem]{Remark}

\newtheorem{corollary}[theorem]{Corollary}

\newtheorem{lemma}[theorem]{Lemma}

\numberwithin{equation}{section} \numberwithin{theorem}{section}

\renewcommand{\topmargin}{0cm}
\renewcommand{\oddsidemargin}{5mm}
\renewcommand{\evensidemargin}{5mm}
\renewcommand{\textwidth}{150mm}
\renewcommand{\textheight}{230mm}

\def\R{\mathbb R}

\def\S{\mathbb S}

\def\na{\nabla}

\def\ir#1{\mathbb R^{#1}}

\def\f#1#2{\frac{#1}{#2}}

\def\a{\alpha}
\def\be{\beta}

\def\r{\Re_{I\!V}}

\def\p#1{\partial #1}

\def\de{\delta}
\def\De{\Delta}
\def\e{\eta}
\def\ep{\epsilon}

\def\G{\Gamma}
\def\g{\gamma}

\def\la{\lambda}

\def\lan{\langle}
\def\ran{\rangle}

\def\Om{\Omega}
\def\th{\theta}

\def\Si{\Sigma}

\def\r{\rho}
\def\z{\zeta}

\begin{document}

\title
[volume growth, eigenvalue and compactness for
self-shrinkers] {volume growth, eigenvalue and
compactness for self-shrinkers}

\author{Qi Ding}\author{Y.L.Xin}
\address{Institute of Mathematics, Fudan University,
Shanghai 200433, China} \email{09110180013@fudan.edu.cn}
\email{ylxin@fudan.edu.cn}
\thanks{The research was partially supported by
NSFC}

\begin{abstract}
In this paper, we show an optimal volume growth for
self-shrinkers, and estimate a lower bound of the first
eigenvalue of $\mathcal{L}$ operator on self-shrinkers, inspired
by the first eigenvalue conjecture on minimal hypersurfaces in
the unit sphere by Yau \cite{SY}. By the eigenvalue estimates, we
can prove a compactness theorem on a class of compact
self-shrinkers in $\ir{3}$ obtained by Colding-Minicozzi under
weaker conditions.
\end{abstract}

\maketitle

\section{Introduction}

Let $X : M^n \rightarrow \R^{n+m}$ be an isometric immersion from
an $n$-dimensional manifold $M^n$ to Euclidean space $\R^{n+m}(m
\ge 1)$ with the tangent bundle $TM$ and the normal bundle $NM$
along $M$. Let $\na$ and $\overline{\na}$ be the Levi-Civita
connections on $M$ and $\R^{n+m}$, respectively.   Then we define
the second fundamental form $B$ by
$B(V,W)=(\overline{\na}_VW)^N=\overline{\na}_VW-\na_VW$ for any
$V,W\in\G(TM)$, where $(\cdots)^N$ stands for the orthogonal
projection into the normal bundle $NM$. The mean curvature vector
$H$ of $M$ is given by
$H=\mathrm{trace}(B)=\sum_{i=1}^nB(e_i,e_i),$ where $\{e_i\}$ is
a local orthonormal frame field of $M$.

$M^n$ is said to be a \emph{self-shrinker} in $\R^{n+m}$ if it satisfies
\begin{equation}\label{0.1}
H= -\frac{X^N}{2}.
\end{equation}
Here, the factor $-\f12$ (when the codimension $m=1$, the
definition here is as the same as \cite{CM1}) could be replaced
by other negative number, while Ecker-Huisken defines $H=-X^N$
\cite{EH}. Self-shrinkers play an important role in the study of
mean curvature flow. They are not only special solutions to
the mean curvature flow equations (those where later time slices
are rescalings of earlier), but they also describe all possible
blow ups at a given type I singularity of a mean curvature flow
(abbreviated by MCF in what follows).

After the pioneer work on self-shrinking hypersurfaces of G.
Huisken \cite{H1}\cite{H2},  T. H. Colding and W. P. Minicozzi II
gave a comprehensive study for self-shrinking hypersurfaces
\cite{CM1}. Their papers reveal the importance of the subject.
For higher codimension, there is a few study, see \cite{Sm} for
example.

There are several other ways to characterize self-shrinkers (see \cite{CM2} for hypersurfaces,
and high codimensional situation is similar):\\
(1) The one-parameter family of submanifolds $\sqrt{-t}M\subset\R^{n+m}$ satisfies MCF equations.\\
(2) $M$ is a minimal submanifold in $\R^{n+m}$ endowed  with the
 conformally flat metric of the conformal factor $e^{-\f{|X|^2}{2n}}.$\\
(3) $M$ is a critical point for the functional $F$ defined on
immersed submanifolds in  $\R^{n+m}$ by
\begin{equation}\label{0.2}
F(M)=(4\pi)^{-n/2}\int_Me^{-\f{|X|^2}4}d\mu.
\end{equation}

Self-shrinkers satisfy  elliptic equations(systems) of the second order,
see (\ref{0.1}). It is an important class of submanifolds, which
is closely related to minimal surface theory. We expect certain
technique in minimal surface theory (see \cite{X2}) could be
modified to study self-shrinkers.

For a complete non-compact manifold the volume growth is
important. By easy arguments we can show that any complete
non-compact self-shrinker properly immersed in Euclidean space with arbitrary
codimension has Euclidean volume growth, just like the trivial
self-shrinker: planes. It is in a sharp contrast to the complete
minimal submanifolds in Euclidean space. Even for complete stable
minimal hypersurfaces, it is still unclear whether they have
Euclidean volume growth.

\begin{theorem}
Any complete non-compact properly immersed  self-shrinker $M^n$
in $\R^{n+m}$ has Euclidean volume growth at most.
\end{theorem}

It is natural to raise a counterpart of the Calabi-Chern
problem on minimal surfaces in $\ir{3}$. Is there a complete
non-compact self-shrinker in Euclidean space, which is contained
in a Euclidean ball?

\begin{remark}

It is worthy to compare the above Theorem with the interesting
result of Cao-Zhou on the volume growth of complete gradient
shrinking Ricci soliton \cite{CZ}. \end{remark}

Let $\Si^n$ be a compact embedded minimal hypersurface in
$(n+1)$-dimensional sphere $\S^{n+1}$. It is well known that the
coordinate functions are eigenfunctions of Laplacian operator on
$\Si$ with eigenvalue $n$. In \cite{SY}, S. T. Yau conjectured
that the first eigenvalue of $\Si$ would be $n$. Choi-Wang,
\cite{CW}, proved that the first eigenvalue of $\Si$ is  bounded
below by $n/2$.

In \cite{CS}, H. I. Choi and R. Schoen gave the compactness theorem
for minimal surfaces using the first
eigenvalue estimates for Laplacian operator on $\Si$. Precisely, let $N$ be a compact 3-dimensional
manifold with positive Ricci curvature, then the space of compact
embedded minimal surfaces of fixed topological type in $N$ is
compact in the $C^k$ topology for any $k\geq2.$

In \cite{CM2}, Colding-Minicozzi  proved a compactness theorem
for complete embedded self shrinkers in $\R^3$. Such compactness
theorem acts a key role for proving a long-standing conjecture
(of Huisken) classifying the singularities of mean curvature flow
starting from a generic closed embedded surface (see \cite{CM1}).

Let $\De$, $\mathrm{div}$ and $d\mu$ be Laplacian, divergence and
volume element on $M$, respectively. There is a  linear operator
$$\mathcal{L}=\Delta-\frac{1}{2}\langle
X,\nabla(\cdot)\rangle=e^{\f{|X|^2}4}\mathrm{div}(e^{-\f{|X|^2}4}\na\cdot).$$
On Euclidean space, this operator is so-called Ornstein-Uhlenbeck
operator in stochastic analysis. So it can be seen as a
generalization of Ornstein-Uhlenbeck operator. The $\mathcal{L}$
operator was introduced and studied firstly on self-shrinkers by
Colding-Minicozzi in  \cite{CM1}, where the authors also showed
that $\mathcal{L}$ is self-adjoint respect to the measure
$e^{-\f{|X|^2}4}d\mu.$ It is a weighted Laplacian and closely
related to the self-shrinkers.

In Euclidean space the eigenvalues of the Ornstein-Uhlenbeck are
well-known. On self-shrinkers it is interesting to study the
eigenvalues of $\mathcal{L}$ operator. Now, we estimate its first
eigenvalue (please see Chapter 3 for the definition) in a manner analogous to the arguments in \cite{CW}.
For compact self-shrinkers the estimates are rather neat. It is
also enough for the compactness applications.



\begin{theorem}\label{firsteigen}
Let $M^n$ be a compact embedded self-shrinker in $\R^{n+1}$, then the first eigenvalue $\lambda_1$ for
the operator $\mathcal{L}$ on $M$ satisfies
$\lambda_1\in[\frac{1}{4},\frac{1}{2}]$.
\end{theorem}

With the help of Theorem \ref{firsteigen}, uniform volume growth for compact embedded self-shrinkers could be estimated by genus and this will yield a compactness
theorem. Give a non-negative integer g and a constant $D>0$, and
let $S_{g,D}$ denote the space of all compact embedded
self-shrinkers in $\R^3$ with genus at most $g$, and diameter at
most $D$.
We have a compactness theorem as follows.
\begin{theorem}\label{compactness2}
For each fixed $g$ and $D$, the space $S_{g,D}$ is compact.
Namely, any sequence in $S_{g,D}$ has a subsequence that
converges uniformly in the $C^k$ topology (for any $k\ge0$) to a
surface in $S_{g,D}$.
\end{theorem}
Colding-Minicozzi gave in \cite{CM2} a compactness theorem for a class of self-shrinkers with bounded entropy in $\R^3$. The assumption of bounded entropy is natural since Colding-Minicozzi proved the conjecture of Huisken in \cite{CM1} and these automatically satisfy such a bound. However, Theorem \ref{compactness2} shows the compactness theorem holds without the assumption of bounded entropy for compact case (please see Corollary 8.2 of \cite{CM1}).

In this paper, we always suppose that $M$ is an $n-$dimensional smooth submanifold in $\R^{n+m}$
with $n\geq2$, the function $\r=e^{-\f{|X|^2}4}$ with $X=(x_1,\cdots,x_{n+m})\in\R^{n+m}$. Let
$\lan\cdot,\cdot\ran$ be standard inner product of $\R^{n+m},$
and $B_r$ be a standard ball in $\R^{n+m}$ with radius $r$ and
centered at the origin, and $D_r=M\cap B_r$ for $r>0$. When $m=1$ (codimension is 1), let $\nu$ be
unit outward normal field of $M$, and $\lan H,\nu\ran$ be mean
curvature of $M$. We also write $H=\lan H,\nu\ran$ if there is no
ambiguity in the context. We agree with the following range of
indices
$$1\le i, j, k, \cdots \le n+m,\quad 1\le \a, \be, \g, \cdots \le n.$$

\section{Volume Growth of Self Shrinkers}

Let $M$ be an $n$-dimensional complete self shrinkers in
$\R^{n+m}$. By \eqref{0.1} we have $\De X=H=-\f12X^N$ for any
$X=(x_1,\cdots,x_{n+m})\in\R^{n+m}$, then(see also \cite{CM1})
\begin{equation}\aligned\label{4.0}
\mathcal{L}x_i=\De x_i-\f12\lan X,\na x_i\ran=-\f12\lan X^N,E_i\ran-\f12\lan X,(E_i)^T\ran=-\f12x_i,
\endaligned
\end{equation}
where $\{E_i\}_{i=1}^{n+m}$ is a standard basis of $\R^{n+m}$ and
$(\cdots)^T$ denotes the orthogonal projection into the tangent
bundle $TM$. Moreover,
\begin{equation}\aligned\label{0.9}
\mathcal{L}|X|^2=2x_i\mathcal{L}x_i+2|\na X|^2=2n-|X|^2,
\endaligned
\end{equation}
and
\begin{equation}\aligned\label{laplace}
\De|X|^2=2\lan X,\De X\ran+2|\na X|^2=2\lan X,H\ran+2n=2n-4|H|^2.
\endaligned
\end{equation}

Now, we give an analytic lemma which will be used in proving
volume growth.
\begin{lemma}
If $f(r)$ is a monotonic increasing nonnegative function on $[0,+\infty)$ with
$f(r)\le C_1r^nf(\f r2)$ on $[C_2,+\infty)$ for some positive constant $n,C_1,C_2$, here $C_2>1$,
 then $f(r)\le C_3e^{2n(\log r)^2}$ on $[C_2,+\infty)$ for some positive constant $C_3$ depending only on
  $n,C_1,C_2,f(C_2)$.
\end{lemma}
\begin{proof} If $f(\f {C_2}2)=0$, then $f(r)=0$ for $r\ge\f {C_2}2$ and this Lemma holds obviously. Hence we could assume $f(\f {C_2}2)>0$, then the function
$g(r)=\log f(r)$ on $[C_2, \infty)$ is well defined. By the
assumption, one has
$$g(r)\le g(\f r2)+\log C_1+n\log r, \qquad \qquad \mathrm{on}\ \ [C_2,+\infty).$$

Let $k=[\f{\log\f {r}{C_2}}{\log 2}]+1$, then $\f {C_2}2\le \f
r{2^k}<C_2,$ which implies $\f{r}{2^{k-1}}\ge C_2>1.$ By
iteration,
\begin{equation}\aligned
g(r)\le& g(\f r{2^2})+2\log C_1+n(\log r+\log \f r2)\le \cdots\\
\le& g(\f r{2^k})+k\log C_1+n\sum_{j=0}^{k-1}\log\f r{2^j}.
\endaligned\nonumber
\end{equation}
Then we have
\begin{equation}\aligned\label{1.1}
g(r)\le& g(C_2)+k(\log C_1+n\log r)\\
\le& g(C_2)+(\f{\log \f{r}{C_2}}{\log 2}+1)(\log C_1+n\log r)\\
\le& \log C_3+2n(\log r)^2,
\endaligned
\end{equation}
where $C_3$ is a positive constant depending only on $n,C_1,C_2,f(C_2)$. By the definition of $g$, \eqref{1.1} implies
$$f(r)\le C_3e^{2n(\log r)^2}$$ on $[C_2,+\infty)$.
\end{proof}

For a complete non-compact $n-$submanifold $M$ in $\R^{n+m}$, we
say that $M$ has \emph{Euclidean volume growth at most} if there
is a constant $C$ so that for all $r\ge 1$,
$$\int_{D_r}1d\mu\le Cr^n.$$
For a complete self-shrinker $M^n$ in $\R^{n+m}$, we define a functional $F_t$ on any set $\Om\subset M$ (see also \cite{CM1} for the definition of $F_t$)
by $$F_t(\Om)=\f1{(4\pi t)^{n/2}}\int_{\Om}
e^{-\f{|X|^2}{4t}}d\mu, \quad \mathrm{for} \quad t>0.$$

\begin{theorem}\label{Vol}
Any complete non-compact properly immersed self-shrinker $M^n$ in
$\R^{n+m}$ has  Euclidean volume growth at most. Precisely,
$\int_{D_r}1d\mu\le Cr^n$ for $r\ge1$, where $C$ is a constant
depending only on $n$ and the volume of $D_{8n}$.
\end{theorem}

\begin{proof}  We
differential $F_t(D_r)$ with respect to $t$,
$$F_t'(D_r)=(4\pi)^{-\f n2} t^{-(\f n2+1)}\int_{D_r} (-\f n2+\f{|X|^2}{4t})e^{-\f{|X|^2}{4t}}d\mu.$$
A straightforward calculation shows
\begin{equation}\aligned\label{1.4}
-e^{\f{|X|^2}{4t}}\mathrm{div}(e^{-\f{|X|^2}{4t}}\na|X|^2)&=-\De|X|^2+\f1{4t}\na|X|^2\cdot\na|X|^2\\
&=-2\langle H,X\rangle-2n+\f1{4t}4|X^T|^2\\
&=|X^N|^2+\f{|X^T|^2}t-2n\\
&\ge\f{|X|^2}t-2n\quad(\ when\ t\ge1\ ),
\endaligned
\end{equation}
where the third equality above uses the self-shrinker's equation
(\ref{0.1}). Since $$\na |X|^2=2X^T$$ and the unit normal
vector to $\p{D_r}$ is $\f{X^T}{|X^T|}$,  then for $t\ge1$ one gets
\begin{equation}\aligned\label{1.2}
F_t'(D_r)\le&\pi^{-\f n2}(4t)^{-(\f n2+1)}\int_{D_r} -\mathrm{div}(e^{-\f{|X|^2}{4t}}\na|X|^2)d\mu\\
=&\pi^{-\f n2}(4t)^{-(\f n2+1)}\int_{\p D_r} -2|X^T|e^{-\f{|X|^2}{4t}}\le0.
\endaligned
\end{equation}
We now integrate $F_t'(D_r)$ over $t$ from 1 to
$r^2\ge1$ and get $F_{r^2}(D_r)\le F_1(D_r)$, namely,
\begin{equation}\aligned\nonumber
(4\pi r^2)^{-\f n2}\int_{D_r}e^{-\f{|X|^2}{4r^2}}d\mu\le (4\pi)^{-\f n2}\int_{D_r}e^{-\f{|X|^2}{4}}d\mu.
\endaligned
\end{equation}
Then
\begin{equation}\aligned\label{1.3}
\f1{r^n}e^{-\f14}\int_{D_r}1d\mu\le&\f1{r^n}\int_{D_r}e^{-\f{|X|^2}{4r^2}}d\mu\le \int_{D_r}e^{-\f{|X|^2}{4}}d\mu=\int_{(D_r\setminus D_{\f r2})\cup D_{\f r2}}e^{-\f{|X|^2}{4}}d\mu\\
\le&e^{-\f{r^2}{16}}\int_{D_r}1d\mu+\int_{D_{\f r2}}1d\mu.
\endaligned
\end{equation}
Let $f(r)=\int_{D_r}1d\mu$, then (\ref{1.3}) implies $$f(r)\le 2e^{\f14}r^nf(\f r2) \quad for \quad r\ge8n.$$
With the help of Lemma 2.1, we obtain $$f(r)\le C_4e^{2n(\log r)^2} \quad for \quad r\ge8n,$$
where $C_4$ is a constant depending only on $n,f(8n)$. Hence
\begin{equation}\aligned\nonumber
\int_M e^{-\f{|X|^2}4}d\mu=&\sum_{j=0}^{\infty}\int_{D_{8n(j+1)}\setminus D_{8nj}}e^{-\f{|X|^2}{4}}d\mu\le\sum_{j=0}^{\infty}e^{-\f{(8nj)^2}4}f(8n(j+1))\\
\le&C_4\sum_{j=0}^{\infty}e^{-\f{(8nj)^2}4}e^{2n(\log (8n)+\log(j+1))^2}\le C_5,
\endaligned
\end{equation}
where $C_5$ is a constant depending only on $n,f(8n)$. Observe (\ref{1.3}), we finish the proof.
\end{proof}

If $M$ is an entire graphic self shrinking hypersurface in
$\R^{n+1}$, then $M$ has Euclidean volume growth. Then by
\cite{EH}, we know that $M$ is a hyperplane, which is also
obtained in \cite{W}.

\begin{remark}
Let $M^n$ be a complete properly immersed self-shrinker in
$\R^{n+m}$, then \eqref{0.9} yields(see also \cite{CM1})
\begin{equation}\aligned\label{mean}
\int_M|X|^2\r=2n\int_M\r.
\endaligned
\end{equation}
For any $0<\ep<\sqrt{2n}$ we obtain
\begin{equation}\aligned\nonumber
2\sqrt{2n}\ep\int_{M\backslash D_{\sqrt{2n}+\ep}}\r&\le\int_{M\backslash D_{\sqrt{2n}+\ep}}(|X|^2-2n)\r=\int_{D_{\sqrt{2n}+\ep}}(2n-|X|^2)\r
\le2n\int_{D_{\sqrt{2n}}}\r.
\endaligned
\end{equation}
Let $\e$ be a cut-off function satisfying $\e\mid_{B_{\sqrt{2n}-\ep}}\equiv1$, $\e\mid_{\R^{n+m}\backslash B_{\sqrt{2n}}}\equiv0$ and $|\bar{\na}\e|\le\f1\ep$, then
\begin{equation}\aligned\nonumber
\sqrt{2n}\ep\int_{D_{\sqrt{2n}-\ep}}\r&\le\int_{D_{\sqrt{2n}}}(2n-|X|^2)\e\r=2\int_{D_{\sqrt{2n}}}\na\e\cdot\f{X^T}{|X|}\r
&\le\f2\ep\int_{D_{\sqrt{2n}}\backslash D_{\sqrt{2n}-\ep}}\r.
\endaligned
\end{equation}
Hence, we conclude that there is a constant $C_6,C_7$ depending only on $n$ and $\ep$, such that
\begin{equation}\aligned\label{estofr}
\int_{M}\r d\mu\le C_6\int_{M\cap (B_{\sqrt{2n}+\ep}\backslash B_{\sqrt{2n}-\ep})}\r d\mu.
\endaligned
\end{equation}
Combining \eqref{1.3} and \eqref{estofr}, we have
$$\int_{D_r}1d\mu\le C_7r^n\int_{M\cap (B_{\sqrt{2n}+\ep}\backslash B_{\sqrt{2n}-\ep})}1d\mu.$$
\end{remark}

\begin{corollary}\label{Ftleq1}
Let $M^n$ be a complete non-compact properly immersed
self-shrinker  in $\R^{n+m}$, then $F_t(M)\le F_1(M)$ for any
$t>0$.
\end{corollary}
\begin{proof}
Let $\z$ is a cut-off function such that
\begin{eqnarray*}
   \z(|X|)= \left\{\begin{array}{ccc}
     1     & \quad\ \ \ {\rm{if}} \ \ \  X\in B_r \\ [3mm]
     linear    & \quad\ \ \ {\rm{if}} \ \ \  X\in B_{2r}\setminus B_r\\ [3mm]
     0  & \quad\quad\  {\rm{if}} \ \ \  X\in \R^{n+m}\setminus B_{2r},
     \end{array}\right.
\end{eqnarray*}
Combining (\ref{1.4}), for any $t>0$ one gets
\begin{equation}\aligned\nonumber
&\left|\int_{D_r} -\mathrm{div}(e^{-\f{|X|^2}{4t}}\na|X|^2)d\mu\right|\\
\le&\left|\int_{M} -\mathrm{div}(e^{-\f{|X|^2}{4t}}\na|X|^2)\z d\mu\right|+\left|\int_{M\setminus D_r} -\mathrm{div}(e^{-\f{|X|^2}{4t}}\na|X|^2)\z d\mu\right|\\
\le&\left|\int_{M}(\na|X|^2\cdot\na\z) e^{-\f{|X|^2}{4t}} d\mu\right|+\int_{M\setminus D_r} \left||X^N|^2+\f{|X^T|^2}t-2n\right|e^{-\f{|X|^2}{4t}}d\mu\\
\le&\f2r\int_{M}|X|e^{-\f{|X|^2}{4t}} d\mu+\int_{M\setminus D_r} \left((1+\f1t)|X|^2+2n\right)e^{-\f{|X|^2}{4t}}d\mu,\\
\endaligned
\end{equation}
which implies
\begin{equation}\aligned\label{1.5}
\lim_{r\rightarrow\infty}\int_{D_r} -\mathrm{div}(e^{-\f{|X|^2}{4t}}\na|X|^2)d\mu=0.
\endaligned
\end{equation}
When $0<t\le1$, from (\ref{1.4}), we have
\begin{equation}\aligned\nonumber
-e^{\f{|X|^2}{4t}}\mathrm{div}(e^{-\f{|X|^2}{4t}}\na_M|X|^2)\le\f{|X|^2}t-2n,
\endaligned
\end{equation}
then
\begin{equation}\aligned\label{1.6}
F_t'(D_r)\ge\pi^{-\f n2}(4t)^{-(\f n2+1)}\int_{D_r} -\mathrm{div}(e^{-\f{|X|^2}{4t}}\na|X|^2)d\mu.\\
\endaligned
\end{equation}
Combining (\ref{1.5}) and (\ref{1.6}), we know
$$F_t'(M)=\lim_{r\rightarrow\infty}F_t'(D_r)\ge0,$$
which implies
$$F_t(M)\le F_1(M),\quad for \ 0<t\le1.$$
On the other hand, by (\ref{1.2}), we have
$$(4\pi R^2)^{-n/2}\int_{D_r}e^{-\f{|X|^2}{4R^2}}d\mu
\le (4\pi)^{-n/2}\int_{D_r}e^{-\f{|X|^2}{4}}d\mu\ \quad\ for\
any\quad R\ge1.$$ Let $r$ go to infinity, thus we finish the
proof.
\end{proof}
In hypersurface case, Colding and Minicozzi II (Lemma 7.10 in \cite{CM1}) has proved a more general result than Corollary \ref{Ftleq1}.

\section{The First Eigenvalue of Self-shrinkers and Compactness Theorem}

Let $\R^{n+1}$ be Euclidean space with the canonical metric, with
Levi-Civita connection $\overline{\nabla}$, Laplacian operator $\overline{\Delta}$, and divergence
$\overline{\mathrm{div}}$. Let
$\overline{\mathcal{L}}=\overline{\Delta}-\frac{1}{2}\langle
X,\overline{\nabla}\cdot\rangle$. Reilly derives a useful integral
formula for Laplacian operator operator \cite{R}(see also \cite{CW}). Now, we
derive a Reilly type formula for the operator $\mathcal{L}$.
\begin{theorem}\label{Reilly}
Let $\Omega$ be a bounded domain in $\R^{n+1}$ with smooth
boundary. Suppose that $f$  satisfies
\begin{eqnarray*}
     \left\{\begin{array}{ccc}
     \overline{\mathcal{L}}f=g     & \quad\ \ \ {\rm{in}} \ \ \  \Omega \\ [3mm]
     f=u    & \quad\quad\  {\rm{on}} \ \ \  \partial\Omega,
     \end{array}\right.
\end{eqnarray*}
where $g$ is a smooth function on $\Omega$ and $u$ is a smooth
function on $\partial\Omega$,  then
$$\int_\Omega
g^2\rho=\int_\Omega|\overline{\nabla}^2f|^2\rho+\frac{1}{2}\int_\Omega|\overline{\nabla}f|^2\rho+2\int_{\partial\Omega}f_\nu\mathcal{L}u\rho-
\int_{\partial\Omega}h(\nabla u,\nabla
u)\rho-\int_{\partial\Omega}f_\nu^2\left(\frac{\langle
X,\nu\rangle}{2}+H\right)\rho,$$ where $h(\cdot,\cdot)=\lan
B(\cdot,\cdot),\nu\ran$, $B$ is the second
fundamental form on $\p\Om$, $\nu$ is the outward unit normal vector
field on $\p\Om$ and mean curvature $H=\mathrm{trace}(h)$.
\end{theorem}
\begin{proof} Let $\{\frac{\partial}{\partial x_i}\}_{i=1}^{n+1}$ be a canonical basis of
$\R^{n+1}$, $f_i=\frac{\partial f}{\partial x_i}$, and so on. Since $\overline{\mathcal{L}}f=g$, then we have
$$\overline{\mathcal{L}}f_i=\sum_j(f_{ijj}-\frac{1}{2}x_jf_{ij})=\frac{\partial}{\partial x_i}(g+\frac{1}{2}\sum_jx_jf_j)-\frac{1}{2}\sum_jx_jf_{ij}=g_i+\frac{f_i}{2},$$
and
\begin{equation}\label{4.1}
\frac{1}{2}\overline{\mathcal{L}}|\overline{\nabla}f|^2=\sum_{i,j}f_{ij}^2+f_i\overline{\mathcal{L}}f_i=|\overline{\nabla}^2f|^2+\langle\overline{\nabla} f,\overline{\nabla} g\rangle+\frac{1}{2}|\overline{\nabla}f|^2.
\end{equation}
Integrating the equality (\ref{4.1}) by parts we get
\begin{equation}\aligned\label{4.2}
\frac{1}{2}\int_\Omega\overline{\mathcal{L}}|\overline{\nabla}f|^2\rho&=\int_\Omega|\overline{\nabla}^2f|^2\rho+\int_\Omega\langle\overline{\nabla} f,\overline{\nabla} g\rangle\rho+\frac{1}{2}\int_\Omega|\overline{\nabla}f|^2\rho\\
&=\int_\Omega|\overline{\nabla}^2f|^2\rho+\frac{1}{2}\int_\Omega|\overline{\nabla}f|^2\rho+\int_\Omega(\overline{\mathrm{div}}(\rho g\overline{\nabla} f)-g\overline{\mathrm{div}}(\rho\overline{\nabla}f))\\
&=\int_\Omega|\overline{\nabla}^2f|^2\rho+\frac{1}{2}\int_\Omega|\overline{\nabla}f|^2\rho+\int_{\partial\Omega} f_\nu g\rho-\int_\Omega g^2\rho.\\
\endaligned
\end{equation}
On the other hand, we select an orthonormal frame field
$\{e_1,\cdots,e_{n+1}\}$ near the boundary of $\Omega$ such that
$\{e_1,\cdots,e_{n}\}$ are tangential to $\partial\Omega$, and
$\nabla_{e_\alpha}e_\beta=\overline{\nabla}_{e_{n+1}}e_i=0$ at a
considered point in $\partial\Omega$  and $\nu=e_{n+1}$ is the
outward unit normal vector. Let $h_{\alpha\beta}=\langle
\overline{\nabla}_{e_\alpha}e_\beta,\nu\rangle=\langle
B(e_\alpha,e_\beta),\nu\rangle$, then integrating by parts gives
\begin{equation}\aligned\label{4.3}
&\frac{1}{2}\int_\Omega\overline{\mathcal{L}}|\overline{\nabla}f|^2\rho
=\frac{1}{2}\int_\Omega\overline{\mathrm{div}}(\rho\overline{\nabla}|\overline{\nabla}f|^2)=\int_{\partial\Omega}\sum_{i=1}^{n+1}(e_if)(e_{n+1}e_if)\rho\\
=&\int_{\partial\Omega}f_\nu(e_{n+1}e_{n+1}f)\rho+\sum_{\alpha=1}^{n}\int_{\partial\Omega}(e_\alpha f)(e_\alpha e_{n+1}f)\rho+\sum_{\alpha=1}^{n}\int_{\partial\Omega}[e_{n+1},e_\alpha](f)(e_\alpha f)\rho\\
=&\int_{\partial\Omega}f_\nu(e_{n+1}e_{n+1}f)\rho-\int_{\partial\Omega}f_\nu(\mathcal{L}u)\rho+\sum_{\alpha,\beta=1}^{n}\int_{\partial\Omega}
h_{\alpha\beta}e_\beta( f)e_\alpha( f)\rho.\\
\endaligned
\end{equation}
Moreover,
\begin{equation}\aligned\label{4.4}
e_{n+1}e_{n+1}f
&=\sum_{i=1}^{n+1}\left(e_ie_if-\left(\overline{\nabla}_{e_i}e_i\right)f\right)
-\sum_{\alpha=1}^n(e_\alpha e_\alpha f)+\sum_{\alpha=1}^n\left(\overline{\nabla}_{e_\alpha}e_\alpha\right) f\\
&=\overline{\Delta}f-\Delta f+\sum_{\alpha=1}^nh_{\alpha\alpha}f_\nu=\overline{\mathcal{L}}f-\mathcal{L}u+\frac{\langle X,\nu\rangle}{2}f_\nu+Hf_\nu.
\endaligned
\end{equation}
Combining (\ref{4.2})-(\ref{4.4}), we complete the proof.
\end{proof}

We would use the above Reilly type formula to estimate the first
eigenvalue of $\mathcal{L}$ operator on a self-shrinker in Euclidean
space. Now the ambient space is not compact. We need the
following  boundary gradient estimate for
$\overline{\mathcal{L}}$.
\begin{lemma}\label{ptest}
Let $\Si$ be a compact embedded hypersurface in $\R^{n+1}$,
$\Omega$ be a bounded domain in $\R^{n+1}$ with $\p\Om=\Si\cup
S_R$. Here $S_R$ is an n-sphere with radius $R$ and centered at the
origin for any $R\ge\sqrt{2(n+1)}+diam(\Si)$. We consider
Dirichlet problem
\begin{eqnarray*}
     \left\{\begin{array}{ccc}
     \overline{\mathcal{L}}f=0     &  \ \ {\rm{in}} \ \ \  \Omega \\ [3mm]
     f\mid_\Si=u\ ,\ f\mid_{S_R}=0 \ ,
     \end{array}\right.
\end{eqnarray*}
where $u$ is a smooth function on $\Si$, then $|\overline{\nabla}
f(X_0)|\leq3\max_{X\in \Si}|u(X)|R$ for any $X_0\in S_R$.
\end{lemma}
\begin{proof} For any $X_0\in S_R$, there is a unique $Y_0\in\R^{n+1}$ such that
$\overline{B_R(0})\cap \overline{B_R(Y_0)}=X_0$. Let $u_0=\max_{X\in \Si}|u(X)|$ and define two barrier functions
$w^\pm(d)=\pm3u_0\left(1-\exp\left(-\frac{d^2-R^2}{2}\right)\right)$, $d(X)=|X-Y_0|$ on the ball $B_{\sqrt{R^2+1}}(Y_0)$.
Now, we prove that the two functions $w^\pm$ satisfy\\
\qquad \ \ (i)\qquad\qquad $\pm\overline{\mathcal{L}}w^\pm<0$\qquad\qquad\quad in $B_{\sqrt{R^2+1}}(Y_0)\cap\Omega$,\\
\qquad \ (ii)\qquad $w^\pm(X_0)=f(X_0)=0$,\\
\qquad (iii)\qquad $w^-(X)\leq f(X)\leq w^+(X), \quad X\in\partial B_{\sqrt{R^2+1}}(Y_0)\cap\Omega$.\\

Let $Y=X-Y_0$, then $d=|Y|$, $\overline{\na} d=\f Y{|Y|}$ and $\overline{\De}d=\f n{|Y|}$, hence
\begin{equation}\aligned
\overline{\mathcal{L}}w^+=(w^+)'\overline{\mathcal{L}}d+(w^+)''|\overline{\na} d|^2=(w^+)''+(w^+)'\left(\f n{|Y|}-\f12\f{X\cdot Y}{|Y|}\right).
\endaligned\nonumber
\end{equation}
Since $(w^+)'=3u_0de^{-\f{d^2-R^2}2}$ and $(w^+)''=3u_0(1-d^2)e^{-\f{d^2-R^2}2}$, then for any $X\in B_{\sqrt{R^2+1}}(Y_0)\cap\Omega$, we have $|X|\le R$, $d=|Y|\ge R$ and
\begin{equation}\aligned
\overline{\mathcal{L}}w^+\le&3u_0(1-d^2)e^{-\f{d^2-R^2}2}+3u_0de^{-\f{d^2-R^2}2}\left(\f nd+\f R2\right)\\
=&3u_0e^{-\f{d^2-R^2}2}\left(1-d^2+n+\f{R}2d\right)\le0\quad (since\ R\ge\sqrt{2(n+1)}\ ).
\endaligned\nonumber
\end{equation}
Thus, (i) is proved. (ii) is obvious. By maximum principle, one can obtain
$$|f(X)|\le u_0, \ \ \mathrm{for\ any} \ X\in\Om.$$
When $X\in\partial B_{\sqrt{R^2+1}}(Y_0)\cap\Omega$, $w^+(X)=3u_0(1-e^{-1/2})\ge f(X)$ and $w^-(X)=-3u_0(1-e^{-1/2})\le f(X)$. Thus, (iii) is proved.

Comparison principle of elliptic equations gives $$w^-(X)\leq f(X)\leq w^+(X), \quad X\in B_{\sqrt{R^2+1}}(Y_0)\cap\Omega.$$
Therefore, the normal derivatives of $w^\pm$ and $f$ satisfy
$$\f{\p w^-}{\p\nu} (X_0)\le\f{\p f}{\p\nu} (X_0)\le\f{\p w^+}{\p\nu} (X_0),$$
which completes the proof.
\end{proof}

We define the first (Neumann) eigenvalue $\la_1$ of the self-adjoint operator
$\mathcal{L}$ in complete self-shrinkers $M^n$ in $\R^{n+1}$ by
$$\lambda_1=\inf_{f\in C^\infty(M)}\left\{\int_M|\nabla f|^2\rho;\quad
\int_Mf^2\rho=1,\int_Mf\rho=0\right\}.$$ By \eqref{4.0}, one
has $\la_1\le\f12$. From the following lemma, $\la_1$ can be arrived
by the first eigenfunction $u$ and $\la_1>0$ for any complete properly immersed self-shrinker.

\begin{lemma}
Let $M^n$ be a complete properly immersed self-shrinker in
$\R^{n+1}$, then there exists a smooth function $u$ with
$\int_Mu^2\rho=1, \int_Mu\rho=0$ such that
$\mathcal{L}u+\lambda_1u=0$ and $\int_M|\nabla
u|^2\rho=\lambda_1$.
\end{lemma}
\begin{proof} By the definition of $\lambda_1$, there exists a sequence $\{f_i\}$ satisfying
\begin{equation}\aligned\label{4.5}
\int_Mf_i^2\rho=1,\int_Mf_i\rho=0 \quad \mathrm{and} \quad \lim_{i\rightarrow\infty}\int_M|\nabla f_i|^2\rho=\lambda_1.
\endaligned
\end{equation}
Since $\lambda_1\leq1/2$, then there exists a $N_0$ such that, for any $i\geq N_0$, $ \int_M|\nabla f_i|^2\rho\leq1.$

Define two Sobolev spaces $L^2(\Om,\rho)$, $H^1(\Om,\rho)$ for any set $\Om\subset M$
by
\begin{equation}\aligned\nonumber
&L^2(\Om,\rho)=\{f\ ;  \int_\Om f^2\r<\infty\ \},\\
&H^1(\Om,\rho)=\{f\ ;  \int_\Om f^2\r<\infty\ \mathrm{and}\ \int_\Om |\na f|^2\r<\infty\},\\
\endaligned
\end{equation}
respectively.
Since $H^1(D_r,\rho)$ is a Hilbert space, then there is a subsequence $\{f_{n_i}\}$ of $\{f_i\}$ converging to some $u_r\in H^1(D_r,\rho)$ weakly, and there is a subsequence $\{f_{n_{k_i}}\}$ of $\{f_{n_i}\}$ converging to some $u_{r+1} \in H^1(D_{r+1},\rho)$ weakly and so on. Hence we could choose a diagonal sequence, denoted by $f_k$ for simplicity, such that $f_k$ converges to some $u_K \in H^1(K,\rho)$ weakly for any compact set $K\subset M$, i.e., we can define $u$ on $M$ such that $u\mid_K=u_K$. By compact embedding theorem, sequence $f_k$ converges to $u$ in the space $L^2(K,\r)$ strongly for any compact set $K$, then
\begin{equation}\aligned\nonumber
\int_Ku^2\rho=\lim_{k\rightarrow\infty}\int_Kf_k^2\rho\leq1.\\
\endaligned
\end{equation}
By weak convergence in $H^1(D_r,\rho)$, one has
\begin{equation}\aligned\nonumber
&\int_K(\nabla u\cdot\na f_k)\rho=\lim_{j\rightarrow\infty}\int_K(\nabla f_j\cdot\na f_k)\rho\\
\leq&\f12\lim_{j\rightarrow\infty}\int_K(|\nabla f_j|^2+|\na f_k|^2)\rho\le\f{\lambda_1}2+\f12\int_K|\na f_k|^2\rho.\\
\endaligned
\end{equation}
Hence
\begin{equation}\aligned\label{4.6}
\int_Mu^2\rho\leq1,\qquad\int_M|\nabla u|^2\rho\leq\lambda_1.\\
\endaligned
\end{equation}
For any sufficiently small $\epsilon>0$ and compact set $K\subset M$, there exists a $k$ such that $\int_K|u-f_k|^2\rho\leq\epsilon$. Combining (\ref{4.5}) and (\ref{4.6}) and Cauchy inequality, we get
\begin{equation}\aligned\label{4.7}
\left|\int_Ku\rho\right|\leq&\int_K|u-f_k|\rho+\left|\int_{K}f_k\rho\right|\leq\int_K|u-f_k|\rho+\int_{M\backslash K}|f_k|\rho\\
\le&\sqrt{\int_K\rho\int_K|u-f_k|^2\rho}+\sqrt{\int_{M\backslash K}\rho\int_{M\backslash K}|f_k|^2\rho}\\
\leq& \sqrt{\ep\int_K\rho}+\sqrt{\int_{M\backslash K}\rho}.
\endaligned
\end{equation}
Since $M$ has Euclidean volume growth, then (\ref{4.7}) implies $$\int_Mu\rho=0.$$

Now let us prove $\int_Mu^2\rho=1$. By Logarithmic type Sobolev
inequalities on self-shrinkers one has  \cite{E}
\begin{equation}\aligned\label{4.8}
\int_M|X|^2f_k^2\rho\leq16\int_M|\nabla f_k|^2\rho+4n\int_Mf_k^2\rho.
\endaligned
\end{equation}
In fact, multiplying a smooth function $g$ with compact support
on the both sides of \eqref{0.9}, then integrating by parts yield
$$\int_M(|X|^2-2n)g^2\r=\int_M(\na|X|^2\cdot\na g^2)\r\le\f12\int_M|X|^2g^2\r+8\int_M|\na g|^2\r.$$
Using such function $g$ to approach $f_k$, we can also get
(\ref{4.8}).

For any $r>0$ (\ref{4.8}) implies
\begin{equation}\aligned\nonumber
r^2\int_{M\setminus D_r}f_k^2\rho\leq16\int_M|\nabla f_k|^2\rho+4n\int_Mf_k^2\rho\leq16+4n,
\endaligned
\end{equation}
then
\begin{equation}\aligned\label{3.1}
\int_{D_r}u^2\rho=\lim_{k\rightarrow\infty}\int_{D_r}f_k^2\rho=1-\lim_{k\rightarrow\infty}\int_{M\setminus D_r}f_k^2\rho\geq1-\frac{16+4n}{r^2}.
\endaligned
\end{equation}
Combining (\ref{4.6}) one arrives at $\int_Mu^2\rho=1$. By the definition
of $\la_1$, we get $\int_M|\nabla u|^2\rho=\lambda_1$. Let us
define a functional
$$I(f)=\int_M|\nabla f|^2\rho-2\lambda_1\int_Mfu\rho$$
and $$\overline{f}=\f{\int_Mf\r}{\int_M\r},$$ then
\begin{equation}\aligned\nonumber
I(f)=&\int_M|\nabla f|^2\rho-2\lambda_1\int_M(f-\overline{f})u\rho\ge\int_M|\nabla f|^2\rho-\lambda_1\int_M\left((f-\overline{f})^2+u^2\right)\rho\\
=&-\la_1+\int_M|\nabla f|^2\rho-\lambda_1\int_M(f-\overline{f})^2\rho\ge-\la_1.
\endaligned
\end{equation}
Since $I(u)=-\la_1$, then the function $u$ arrives at the minimum of the functional $I(\cdot)$. Hence $\frac{\partial}{\partial \epsilon}\mid_{\epsilon=0}I(u+\epsilon\varphi)=0$ for any $\varphi\in C^\infty_c(M)$. By a simple calculation, we have
$$\int_M(\mathcal{L}u+\lambda_1u)\varphi\rho=0.$$
By the regularity theory of elliptic equations $u$ is a smooth
function (see \cite{GT} for example). We finish the proof.
\end{proof}

Now, we give a uniformly positive lower bound of $\la_1$ for compact
embedded self shrinkers in $\R^{n+1}$.

\noindent {\it \textbf{Proof of Theorem \ref{firsteigen}}.} We have known $0<\la_1\le\f12$ in the previous discussion.
Let $B_R$ be an $n$-ball with radius $R$ and centered at the origin, then there is a $R_0$ such that $M\subset\subset B_{R_0}$. Set $\Omega_R$ be a bounded domain in $\R^{n+1}$ with $\p\Om_R=M\cup \p B_R$ for $R\ge R_0$. We consider the following Dirichlet problem
\begin{eqnarray*}
     \left\{\begin{array}{ccc}
     \overline{\mathcal{L}}f=0     &  \ \ {\rm{in}} \ \ \  \Omega_R \\ [3mm]
     f\mid_M=u\ ,\ f\mid_{\p B_R}=0 \ ,
     \end{array}\right.
\end{eqnarray*}
where $u$ is the first eigenfunction of the self-adjoint operator $\mathcal{L}$ in $M$, i.e., $\mathcal{L}u+\lambda_1u=0$ and $\int_Mu^2\rho=1$. By Lemma \ref{ptest}, we get $|\overline{\nabla} f(Y)|\leq3\max_{X\in M}|u(X)|R$ for any $Y\in \p B_R$. Integrating by parts gives
\begin{equation}\aligned\label{4.9}
\int_Mf_\nu\mathcal{L}u\r=-\la_1\int_Mf_\nu u\r=-\la_1\int_{\Om_R} \overline{\mathrm{div}}(\r f\overline{\na} f)=-\la_1\int_{\Om_R}|\overline{\nabla}f|^2\rho,
\endaligned
\end{equation}
Combining \eqref{0.1}, (\ref{4.9}) and Theorem \ref{Reilly}, we have
\begin{equation}\aligned\label{4.10}
0\ge&\int_{\Omega_R}|\overline{\nabla}^2f|^2\rho+\frac{1}{2}\int_{\Omega_R}|\overline{\nabla}f|^2\rho-2\la_1\int_{\Om_R}|\overline{\nabla}f|^2\rho-
\int_{M}h(\nabla u,\nabla u)\rho-\int_{\p B_R}f_\nu^2\f R2\rho\\
\ge&\int_{\Omega_R}|\overline{\nabla}^2f|^2\rho+(\frac{1}{2}-2\la_1)\int_{\Om_R}|\overline{\nabla}f|^2\rho-\int_{M}h(\nabla u,\nabla u)\rho-\f92\max_{X\in M}|u(X)|^2R^3\int_{\p B_R}\r.
\endaligned
\end{equation}
We may assume $\int_{M}h(\nabla u,\nabla u)\rho\le0$, or else we consider the bounded domain $U$ with $\p U=M$ instead of $\Om_R$. By trace theorems in Sobolev spaces (see \cite{HT} for example), there is a positive constant $C$ depending only on $n,R_0$ and $M$ such that
$$\int_{\Omega_{R_0}}|\overline{\nabla}^2f|^2+\int_{\Omega_{R_0}}|\overline{\nabla}f|^2\ge C\int_M|\overline{\na} f|^2.$$ Then for $R>R_0$ one has
\begin{equation}\aligned\label{RR0}
\int_{\Omega_{R}}|\overline{\nabla}^2f|^2\r+\int_{\Omega_{R}}&|\overline{\nabla}f|^2\r
\ge\int_{\Omega_{R_0}}|\overline{\nabla}^2f|^2\r+\int_{\Omega_{R_0}}|\overline{\nabla}f|^2\r\\
&\ge e^{-\f{R_0^2}4}\left(\int_{\Omega_{R_0}}|\overline{\nabla}^2f|^2+\int_{\Omega_{R_0}}|\overline{\nabla}f|^2\right)\ge e^{-\f{R_0^2}4}C\int_M|\overline{\na} f|^2\\
&\ge e^{-\f{R_0^2}4}C\int_M|\na u|^2\r=e^{-\f{R_0^2}4}C\la_1>0.
\endaligned
\end{equation}
If $\la_1<\f14$, then let $R$ go to infinite in (\ref{4.10}), which deduces a contradiction by \eqref{RR0}. Hence we get $\la_1\ge\f14$. \qed

\begin{corollary}
Let $M^n$ be a compact embedded self-shrinker in $\R^{n+1}$, then for any $f\in C^1(M)$ there is a Poincar$\acute{e}$ inequality
$$\int_M (f-\overline{f})^2\rho\leq4\int_M |\nabla f|^2\rho,$$
where $\overline{f}=\f{\int_Mf\r}{\int_M\r}$.
\end{corollary}
\begin{proof} If $f$ is not a constant, let $g=(\int_Mf^2\r-\overline{f}^2\cdot\int_M\r)^{-1/2}(f-\overline{f})$, then $\int_Mg\r=0$ and $\int_Mg^2\r=1$. Since $\la_1$ is the first eigenvalue of self-adjoint operator $\mathcal{L}$, then combining Theorem \ref{firsteigen} we complete the proof.
\end{proof}

Now, let us recall a classical result of P. Yang and S.T. Yau
\cite{YY}.
\begin{theorem}\label{YY}
(Yang-Yau)
Let $(\Si_g^2,ds^2)$ be an orientable Riemann surface of genus $g$ with area $Area(\Si_g)$. Then we have
$$\la_1(\Si_g)\le\f{8\pi(1+g)}{Area(\Si_g)},$$
where $\la_1(\Si_g)$ is the first eigenvalue of $\Delta$ on $\Si_g$.
\end{theorem}

Let $N^3=(\R^3,(\r\de_{ij}))$ and $M^2$ be a compact embedded
self-shrinker in $\R^3$. Let
$\widetilde{g}=\widetilde{g}_{ij}d\th_id\th_j$, $\widetilde{\na}$
and $\widetilde{\De}$ be the metric, the Levi-Civita connection
and the Laplacian operator of $M$ induced from $N^3$, respectively. Denote
the self-shrinker $M$ with metric $\widetilde{g}$ by
$\widetilde{M}$. Let $g_{ij}d\th_id\th_j$ and $d\mu$ be the
metric and the volume element of $M$ induced from $\R^3$, then
$\widetilde{g}_{ij}=\r g_{ij}$. Denote the first eigenvalue of
$\widetilde{M}$ by $\la_1(\widetilde{M})$, then
\begin{equation}\aligned\label{4.11}
\la_1(\widetilde{M})=&\inf_{\int_{M}f\r=0}\f{\int_{M}|\widetilde{\na}f|^2\r d\mu}{\int_{M}f^2\r d\mu}=\inf_{\int_{M}f\r=0}\f{\int_{M}\widetilde{g}^{ij}f_if_j\r d\mu}{\int_{M}f^2\r d\mu}\\
\ge&\inf_{\int_{M}f\r=0}\f{\int_{M}g^{ij}f_if_j\r
d\mu}{\int_{M}f^2\r d\mu}=\la_1\ge\f14,
\endaligned
\end{equation}
where we have used Theorem \ref{firsteigen} in the last inequality in
(\ref{4.11}).

\begin{corollary}\label{grow}
Let $M$ be a compact embedded self-shrinker in $\R^3$ with genus $g$, then
$$\int_{M}\rho\le32\pi(1+g),$$
moreover, $$\int_{D_r}1d\mu\leq 32e^{\f14}\pi(1+g)r^2 \quad
\mathrm{for}\quad r\ge1.$$
\end{corollary}
\begin{proof}
Combining (\ref{4.11}) and Theorem \ref{YY}, we get
\begin{equation}\aligned\label{4.12}
\int_{M}\rho\le32\pi(1+g).
\endaligned
\end{equation}
Combining (\ref{1.3}) and (\ref{4.12})  for $r\ge1$ gives
\begin{equation}\aligned\label{4.15}
\f1{r^2}e^{-\f14}\int_{D_r}1d\mu\le\f1{r^2}\int_{D_r}e^{-\f{|X|^2}{4r^2}}d\mu\le\int_{D_r}e^{-\f{|X|^2}4}d\mu
\le32\pi(1+g),
\endaligned
\end{equation}
which yields
\begin{equation}\aligned\label{4.55}
\int_{D_r}1d\mu\leq 32e^{\f14}\pi(1+g)r^2, \quad
\mathrm{for}\quad r\ge1.
\endaligned
\end{equation}
\end{proof}

For a non-negative integer g and a constant $D>0$, let $S_{g,D}$
denote the space of all compact embedded self-shrinkers in $\R^3$
with genus at most $g$, and diameter at most $D$. Now, we are in
a position to prove a compactness theorem.

\noindent {\it \textbf{Proof of Theorem \ref{compactness2}}.}
For any compact surface $\Si$ in $\R^3$, using Gauss-Bonnet formula one has
\begin{equation}\aligned\label{4.20}
\int_\Si|B|^2=\int_\Si H^2-2\int_\Si K=\int_\Si H^2-4\pi\chi(\Si),
\endaligned
\end{equation}
where $\chi(\Si)$ is the Euler number of surface $\Si$. If
$\Si\in S_{g,D}$, then \eqref{mean} implies there is a $X\in\Si$ with $|X|=\sqrt{2n}$ and $\Si\subset B_{D+\sqrt{2n}}$.
Combining (\ref{laplace}) and \eqref{4.55} gives
\begin{equation}\aligned\label{4.56}
\int_\Si H^2=\int_{\Si}1d\mu\leq 32e^{\f14}\pi(1+g)(D+\sqrt{2n})^2.
\endaligned
\end{equation}
Then \eqref{4.20} and \eqref{4.56} implies
\begin{equation}\aligned\label{4.21}
\int_\Si|B|^2=\int_\Si H^2-4\pi(2-2g)\le32e^{\f14}\pi(1+g)(D+\sqrt{2n})^2+8\pi(g-1).
\endaligned
\end{equation}
According to Proposition 5.10 of \cite{CM3}(see also \cite{CS}
and \cite{CM2}), we complete the proof. \qed

\medskip

\bibliographystyle{amsplain}

\end{document}